\documentclass[final]{siamltex}

\usepackage[usenames]{color}
\usepackage[color]{showkeys}
\definecolor{refkey}{gray}{0.4}
\definecolor{labelkey}{gray}{0.3}

\usepackage[left=1in,top=1in,right=1in,bottom=1in,dvips,letterpaper]{geometry}
\usepackage{setspace,url}
\usepackage[leqno]{amsmath}
\usepackage{amsxtra, amsfonts,amscd,  amssymb, graphicx,subfigure}

\usepackage[algo2e, vlined,ruled]{algorithm2e}
\numberwithin{equation}{section}

\newcommand{\be}{\begin{equation}}
\newcommand{\ee}{\end{equation}}
\newcommand{\ba}{\begin{array}}
\newcommand{\ea}{\end{array}}
\newcommand{\bea}{\begin{eqnarray}}
\newcommand{\eea}{\end{eqnarray}}
\newcommand{\beaa}{\begin{eqnarray*}}
\newcommand{\eeaa}{\end{eqnarray*}}

\newcommand{\half}{\frac{1}{2}}
\newcommand{\br}{\mathbb{R}}

\newcommand{\mtn}{{m \times n}}

\usepackage[usenames]{color}
\usepackage[normalem]{ulem}

\newcommand{\LCal}{\mathcal{L}}

\newcommand{\Shrink}{\mathrm{Shrink}}

\newcommand{\Tr}{\mathbf{Tr}}

\newcommand{\etal}{{et al. }}

\newcommand{\st}{\mbox{ s.t. }}

\newcommand{\Prox}{\mathrm{prox}}
\newcommand{\argmin}{\mathrm{argmin}}
\newcommand\independent{\protect\mathpalette{\protect\independenT}{\perp}}
\def\independenT#1#2{\mathrel{\rlap{$#1#2$}\mkern2mu{#1#2}}}
\usepackage{algorithmic,algorithm}
\newtheorem{remark}[theorem]{Remark}

\begin{document}

\title{Alternating Direction Methods for Latent Variable Gaussian Graphical Model Selection}

\author{Shiqian Ma \footnotemark[1] \and Lingzhou Xue \footnotemark[2] \and Hui Zou \footnotemark[3]}

\renewcommand{\thefootnote}{\fnsymbol{footnote}}
\footnotetext[1]{Institute for Mathematics and Its Applications, 400 Lind Hall, 207 Church Street SE, University of Minnesota, Minneapolis, MN 55455, USA.
\quad Email: maxxa007@ima.umn.edu.}
\footnotetext[2]{Department of Operations Research and Financial Engineering, Princeton University, Princeton, NJ 08536, USA. \quad Email: lzxue@princeton.edu}
\footnotetext[3]{School of Statistics, University of Minnesota, Minneapolis, MN 55455, USA. \quad Email: hzou@stat.umn.edu}
\maketitle

\begin{abstract}
Chandrasekaran, Parrilo and Willsky (2010) proposed a convex optimization problem to characterize graphical model selection in the presence of unobserved variables. This convex optimization problem aims to estimate an inverse covariance matrix that can be decomposed into a sparse matrix minus a low-rank matrix from sample data. Solving this convex optimization problem is very challenging, especially for large problems. In this paper, we propose two alternating direction methods for solving this problem. The first method is to apply the classical alternating direction method of multipliers to solve the problem as a consensus problem. The second method is a proximal gradient based alternating direction method of multipliers. Our methods exploit and take advantage of the special structure of the problem and thus can solve large problems very efficiently. Global convergence result is established for the proposed methods. Numerical results on both synthetic data and gene expression data show that our methods usually solve problems with one million variables in one to two minutes, and are usually five to thirty five times faster than a state-of-the-art Newton-CG proximal point algorithm.
\end{abstract}

\begin{keywords}Alternating Direction Method, Proximal Gradient, Global Convergence,
Gaussian Graphical Models, Latent Variables, Sparsity, Low-rank, Regularization.
\end{keywords}

\section{Introduction}\label{sec:intro}

In this paper, we consider alternating direction methods with the theoretical guarantee of global convergence for computing the latent-variable graphical model selection \cite{Chandrasekaran-Parrilo-Willsky-2010}. Graphical model selection is closely related to the inverse covariance matrix estimation problem, which is of fundamental importance in multivariate statistical inference. In particular, when data $X=(X_1,\cdots,X_p)'$ follow a $p$-dimensional joint normal distribution with some unknown variance matrix $\Sigma$, the precision matrix $\Theta=\Sigma^{-1}$ can be directly translated into a Gaussian graphical model. The zero entries in the precision matrix $\Theta=\bigl(\theta_{ij}\bigr)_{1\le i,j\le p}$ precisely capture the desired conditional independencies in the Gaussian graphical model \cite{lauritzen1996,edwards2000}, i.e.
$
\theta_{ij}=0
$
if and only if
$
X_i \independent X_j |~X_{-(i,j)}.
$
The Gaussian graphical model has been successfully used to explore complex systems consisting of Gaussian random variables in many research fields, including gene expression genomics \cite{nfriedman2004,wille2004}, image processing \cite{lisz2009}, macroeconomics determinants study \cite{dobra2009}, and social study \cite{ahmed2009,kolar2010}.

Nowadays, massive high-dimensional data are being routinely generated with rapid advances of modern high-throughput technology (e.g. microarray and functional magnetic resonance imaging). Estimation of a sparse graphical model has become increasingly important in the high-dimensional regime, and some well-developed penalization techniques have received considerable attention in the statistical literature. \cite{meinshausen2006} was the first to study the high-dimensional sparse graphical model selection problem, and they proposed the neighborhood penalized regression scheme which performs the lasso \cite{tibshirani1996} to fit each neighborhood regression and summarizes the sparsity pattern by aggregation via union or intersection. \cite{peng2009} proposed the joint sparse regression model to jointly estimate all neighborhood lasso penalized regressions. \cite{yuan2010} considered the Dantzig selector \cite{candes2007} as an alternative to the lasso in each neighborhood regression. \cite{clime11} proposed a constrained $\ell_1$ minimization estimator called CLIME for estimating sparse precision matrices, and established rates of convergence under both the entrywise $\ell_\infty$ norm and the Frobenius norm. Computationally CLIME can be further decomposed into a series of vector minimization problems.

The $\ell_1$-penalized maximum normal likelihood method is another popular method for graphical model selection \cite{yuan2007,banerjee2008,friedman2008,rothman2008}. \cite{rothman2008} established its rate of convergence under the Frobenius norm. Under the irrepresentable conditions, \cite{ravikumar2008} obtained the convergence rates under the entrywise $\ell_\infty$ norm and the spectral norm. Define the entrywise $\ell_1$ norm of $S$ as the sum of absolute values of the entries of $S$, i.e., $\|S\|_1:=\sum_{ij}|S_{ij}|$. For a given sample covariance matrix $\hat\Sigma\in\br^{p\times p}$, the $\ell_1$-penalized maximum normal likelihood estimation can be formulated as the following convex optimization problem.
\be
\label{prob:Glasso} \min_S \ \langle S,\hat\Sigma\rangle - \log\det S + \rho\|S\|_1,
\ee
where the first part $\langle S,\hat\Sigma\rangle - \log\det S$ gives the normal log-likelihood function of $S$, and the entrywise $\ell_1$ norm $\|S\|_1$ is used to promote the sparsity of the resulting matrix. Note that in the literature the $\ell_1$-penalized maximum normal likelihood usually uses the so-called 1-off absolute penalty $\|S\|_{1,{\rm off}}:=\sum_{i \ne j}|S_{ij}|$. However, $\|S\|_1$ and $\|S\|_{1,{\rm off}}$ cause no difference when using our algorithm. \cite{Chandrasekaran-Parrilo-Willsky-2010} have used $\|S\|_1$ in defining their convex optimization problem and hence we follow their convention in the current paper.

The aforementioned Gaussian graphical model selection methods were proposed under the ideal setting without missing variables. The recent paper by \cite{Chandrasekaran-Parrilo-Willsky-2010} considered a more realistic scenario where the full data consist of both observed variables and missing (hidden) variables. Let $X_{p\times 1}$ be the observed variables. Suppose that there are some hidden variables $Y_{r\times 1}$ ($r\ll p$) such that $(X,Y)$ jointly follow a multivariate normal distribution. Denote the covariance matrix by $\Sigma_{(X,Y)}$ and the precision matrix by $\Theta_{(X,Y)}$. Then we can write $\Sigma_{(X,Y)}=[\Sigma_X,\Sigma_{XY};\Sigma_{YX},\Sigma_Y]$ and $\Theta_{(X,Y)}=[\Theta_X,\Theta_{XY};\Theta_{YX},\Theta_Y]$. Given the hidden variables $Y$, the conditional concentration matrix of observed variables, $\Theta_X$, is sparse for a sparse graphical model. However, the marginal concentration matrix of observed variables, $\Sigma_X^{-1}=\Theta_X-\Theta_{XY}\Theta_{Y}^{-1}\Theta_{YX}$, might not be a sparse matrix but a difference between the sparse term $\Theta_X$ and the low-rank term $\Theta_{XY}\Theta_{Y}^{-1}\Theta_{YX}$. The problem of interest is to recover the sparse conditional matrix $\Theta_X$ based on observed variables $X$. \cite{Chandrasekaran-Parrilo-Willsky-2010} accomplished this goal by solving a convex optimization problem under the assumption that $\Sigma_X^{-1}=S-L$ for some sparse matrix $S$ and low-rank matrix $L$. The low rank assumption on $L$ holds naturally since $r$ is much less than $p$. Motivated by the success of the convex relaxation for rank-minimization problem, \cite{Chandrasekaran-Parrilo-Willsky-2010} introduced a regularized maximum normal likelihood decomposition framework called the latent variable graphical model selection (LVGLASSO) as follows.
\be\label{prob:latent-variable-original}
\min_{S,L} \ \langle S-L, \hat\Sigma_X \rangle - \log\det(S-L) + \alpha\|S\|_1 + \beta\Tr(L), \quad \st \ S-L\succ 0, L\succeq 0,
\ee
where $\hat\Sigma_X$ is the sample covariance matrix of $X$ and $\Tr(L)$ denotes the trace of matrix $L$. In the high-dimensional setting, \cite{Chandrasekaran-Parrilo-Willsky-2010} established the consistency theory for \eqref{prob:latent-variable-original} concerning its recovery of the support and sign pattern of $S$ and the rank of $L$.

Solving the convex optimization problem \eqref{prob:latent-variable-original} is very challenging, especially for large problems. \cite{Chandrasekaran-Parrilo-Willsky-2010} considered \eqref{prob:latent-variable-original} as a log-determinant semidefinite programming (SDP) problem, and used a Newton-CG based proximal point algorithm (LogdetPPA) proposed by \cite{Wang-Sun-Toh-2009} to solve it. However, LogdetPPA does not take advantage of the special structure of the problem, and we argue that it is inefficient for solving large-scale problems. To illustrate our point, let us consider the special case of \eqref{prob:latent-variable-original} with $L=0$, and then the latent variable graphical model selection \eqref{prob:latent-variable-original} exactly reduces to the Gaussian graphical model selection \eqref{prob:Glasso}. Note that \eqref{prob:Glasso} can be rewritten as
\[
\min_S \ \max_{\|W\|_\infty\leq\rho} -\log\det S+\langle\hat\Sigma_X+W,X\rangle,
\] where $\|W\|_\infty$ is the largest absolute value of the entries of $U$. The dual problem of \eqref{prob:Glasso} can be obtained by exchanging the order of max and min, i.e.,
\[
\max_{\|W\|_\infty\leq\rho}\min_S \ -\log\det X+\langle\hat\Sigma_X+W,S\rangle,
\] which is equivalent to
\be\label{prob:Glasso-dual}
\max_W\ \{ \log\det W +p :  \|W-\hat\Sigma_X\|_\infty\leq\rho\}.
\ee
Both the primal and the dual graphical Lasso problems \eqref{prob:Glasso} and \eqref{prob:Glasso-dual} can be viewed as semidefinite programming problems and can be solved via interior point methods (IPMs) in polynomial time \cite{BV-ConvexBook2004}. However, the per-iteration computational cost and memory
requirements of an IPM are prohibitively high for \eqref{prob:Glasso} and \eqref{prob:Glasso-dual}, especially when the size of the matrix is large.
Customized SDP based methods such as the ones studied in \cite{Wang-Sun-Toh-2009} and \cite{Li-Toh-2010} require a reformulation of the problem that increases the size of the problem and thus makes them impractical for solving large-scale problems.
Therefore, most of the methods developed for solving \eqref{prob:Glasso} and \eqref{prob:Glasso-dual} are first-order methods. These methods include block coordinate descent type methods \cite{banerjee2008,friedman2008,Scheinberg-Rish-2009,Witten11}, projected gradient method \cite{Duchi-UAI-2008} and variants of Nesterov's accelerated method \cite{Aspremont-Banerjee-ElGhaoui-2008,Lu-covsel-siopt-2009}. Recently, alternating direction methods have been applied to solve \eqref{prob:Glasso} and shown to be very effective \cite{Yuan-2009,Scheinberg-Ma-Goldfarb-NIPS-2010}.

In this paper, we propose two alternating direction type methods to solve the latent variable graphical model selection. The first method is to apply the alternating direction method of multipliers to solve this problem. This is due to the fact that the latent variable graphical model selection can be seen as a special case of the consensus problem discussed in \cite{Boyd-etal-ADM-survey-2011}. The second method we propose is an alternating direction method with proximal gradient steps. To apply the second method, we first group the variables into two blocks and then apply the alternating direction method with one of the subproblems being solved inexactly by taking a proximal gradient step. Our methods exploit and take advantage of the special structure of the problem and thus can solve large problems very efficiently. Although the convergence results of the proposed methods are not very different from the existing results for alternating direction type methods, we still include the convergence proof for the second method in the appendix for completeness. We apply the proposed methods to solving problems from both synthetic data and gene expression data and show that our method outperform the state-of-the-art Newton-CG proximal point algorithm LogdetPPA significantly on both accuracy and CPU times.

The rest of this paper is organized as follows. In Section \ref{sec:preliminaries}, we give some preliminaries on alternating direction method of multipliers and proximal mappings. In Section \ref{sec:consensus}, we propose solving LVGLASSO \eqref{prob:latent-variable-original} as a consensus problem using the classical alternating direction method of multipliers. We propose the proximal gradient based alternating direction method for solving \eqref{prob:latent-variable-original} in Section \ref{sec:ADMM}. In Section \ref{sec:numerical}, we apply our alternating direction method to solving \eqref{prob:latent-variable-original} using both synthetic data and gene expression data. We draw some conclusions in Section \ref{sec:conclusion}.

\section{Preliminaries}\label{sec:preliminaries}

Problem \eqref{prob:latent-variable-original} can be rewritten in the following equivalent form by introducing a new variable $R$:
\be\label{prob:latent-variable-original-R}
\ba{ll} \min & \langle R, \hat\Sigma_X \rangle - \log\det R + \alpha\|S\|_1 + \beta\Tr(L) \\
\st & R = S - L, R \succ 0, L \succeq 0, \ea\ee
which can be further reduced to
\be\label{prob:latent-variable}  \min \ \langle R, \hat\Sigma_X \rangle - \log\det R + \alpha\|S\|_1 + \beta\Tr(L) + \mathcal{I}(L\succeq 0),\quad \st \ R -S + L =0,\ee where the indicator function $\mathcal{I}(L\succeq 0)$ is defined as
\be \label{indicator-function} \mathcal{I}(L\succeq 0) := \left\{\ba{ll} 0, & \mbox{ if } L\succeq 0 \\
                                                                        +\infty, & \mbox{ otherwise.}\ea\right. \ee
Note that we have dropped the constraint $R\succ 0$ since it is already implicitly imposed by the $\log\det R$ function.

Now since the objective function involves three separable convex functions and the constraint is simply linear, Problem \eqref{prob:latent-variable} is suitable for alternating direction method of multipliers (ADMM). ADMM is closely related to the Douglas-Rachford and Peaceman-Rachford operator-splitting methods for finding zero of the sum of two monotone operators that have been studied extensively in \cite{Douglas-Rachford-56,Peaceman-Rachford-55,Lions-Mercier-79,Eckstein-thesis-89,Eckstein-Bertsekas-1992,Combettes-Pesquet-DR-2007,Combettes-Wajs-05}.
ADMM has been revisited recently due to its success in the emerging applications of structured convex optimization problems arising from image processing, compressed sensing, machine learning, semidefinite programming and statistics etc.
(see e.g., \cite{Glowinski-LeTallec-89,Gabay-83,Wang-Yang-Yin-Zhang-2008,Yang-Zhang-2009,Goldstein-Osher-08,Qin-Goldfarb-Ma-2011,Yuan-2009,Scheinberg-Ma-Goldfarb-NIPS-2010,
Goldfarb-Ma-Ksplit,Goldfarb-Ma-Scheinberg-2010,
Malick-Povh-Rendl-Wiegele-2009,Wen-Goldfarb-Yin-2009,Boyd-etal-ADM-survey-2011,Boyd-NIPS-2011,Ma-SPCA-2011-submit}).

Problem \eqref{prob:latent-variable} is suitable for alternating direction methods because the three convex functions involved in the objective function, i.e.,
\be \label{function-R} f(R):= \langle R, \hat\Sigma_X \rangle - \log\det R,\ee
\be \label{function-S} g(S):= \alpha\|S\|_1, \ee
and
\be \label{function-L} h(L):= \beta\Tr(L) + \mathcal{I}(L\succeq 0), \ee
have easy proximal mappings. Note that the proximal mapping of function $c:\br^\mtn\rightarrow \br^\mtn$ for given $\xi>0$ and $Z\in\br^\mtn$ is defined as
\be \label{def:proximal-mapping} \Prox(c,\xi,Z):= \argmin_{X\in\br^\mtn} \ \frac{1}{2\xi}\|X-Z\|_F^2 + c(X).\ee The proximal mapping of $f(R)$ defined in \eqref{function-R} is
\be\label{prox-f-R}\Prox(f,\xi,Z) := \argmin_R \ \frac{1}{2\xi}\|R-Z\|_F^2 + \langle R, \hat\Sigma_X \rangle - \log\det R.\ee The first-order optimality conditions of \eqref{prox-f-R} are given by
\be\label{prox-f-R-optcond} R + \xi\hat{\Sigma}_X - Z - \xi R^{-1} = 0. \ee
It is easy to verify that
\be\label{prox-f-R-sol} R:= U\diag(\gamma)U^\top,\ee
satisfies \eqref{prox-f-R-optcond} and thus gives the optimal solution of \eqref{prox-f-R}, where $U\diag(\sigma)U^\top$ is the eigenvalue decomposition of matrix $\xi\hat{\Sigma}_X-Z$ and
\be\label{prox-f-R-sol-eigen} \gamma_i=\left(-\sigma_i+\sqrt{\sigma_i^2+4\xi}\right)/2, \forall i=1,\ldots,p.\ee Note that \eqref{prox-f-R-sol-eigen} guarantees that the solution of \eqref{prox-f-R} given by \eqref{prox-f-R-sol} is a positive definite matrix. The proximal mapping of $g(S)$ defined in \eqref{function-S} is
\be\label{prox-g-S}\Prox(g,\xi,Z) := \argmin_S \ \frac{1}{2\xi}\|S-Z\|_F^2 + \alpha\|S\|_1. \ee
It is well known that \eqref{prox-g-S} has a closed-form solution that is given by the $\ell_1$ shrinkage operation
\[ S^{k+1}:=\Shrink(Z,\alpha\xi), \]
where $\Shrink(\cdot,\cdot)$ is defined as
\be\label{def-L1-shrinkage} [\Shrink(Z,\tau)]_{ij} := \left\{\ba{ll} Z_{ij} - \tau, & \mbox{ if } Z_{ij} > \tau \\
                                                                    Z_{ij} + \tau, & \mbox{ if } Z_{ij} < -\tau \\
                                                                    0           , & \mbox{ if } -\tau \leq Z_{ij}\leq \tau. \ea\right. \ee
The proximal mapping of $h(L)$ defined in \eqref{function-L} is
\be\label{prox-h-L}\Prox(h,\xi,Z) := \argmin_L \ \frac{1}{2\xi}\|L-Z\|_F^2 + \beta\Tr(L) + \mathcal{I}(L\succeq 0).\ee It is easy to verify that the solution of \eqref{prox-h-L} is given by
\be\label{prox-h-L-sol} L:=U\diag(\gamma)U^\top,\ee
where
$Z=U\diag(\sigma)U^\top$ is the eigenvalue decomposition of $Z$ and $\gamma$ is given by
\be\label{prox-h-L-sol-eigen} \gamma_i := \max\{\sigma_i - \xi\beta ,0\}, \quad i=1,\ldots,p. \ee
Note that \eqref{prox-h-L-sol-eigen} guarantees that $L$ given in \eqref{prox-h-L-sol} is a positive semidefinite matrix.

The discussions above suggest the following natural ADMM for solving \eqref{prob:latent-variable} be efficient.
\be\label{alg:ADM-3} \left\{\ba{lll} R^{k+1} & := & \argmin_R \ \LCal_\mu(R,S^k,L^k;\Lambda^k) \\
                                     S^{k+1} & := & \argmin_S \ \LCal_\mu(R^{k+1},S,L^k;\Lambda^k) \\
                                     L^{k+1} & := & \argmin_L \ \LCal_\mu(R^{k+1},S^{k+1},L;\Lambda^k) \\
                                     \Lambda^{k+1} & := & \Lambda^k - (R^{k+1}-S^{k+1}+L^{k+1})/\mu,\ea\right. \ee
where the augmented Lagrangian function is defined as
\be \label{aug-lag-func} \LCal_\mu(R,S,L;\Lambda):= \langle R, \hat\Sigma_X \rangle - \log\det R + \alpha\|S\|_1 + \beta\Tr(L) + \mathcal{I}(L\succeq 0) - \langle \Lambda,R-S+L \rangle + \frac{1}{2\mu}\|R-S+L\|_F^2, \ee $\Lambda$ is the Lagrange multiplier and $\mu>0$ is the penalty parameter. Note that the three subproblems in \eqref{alg:ADM-3} correspond to the proximal mappings of $f$, $g$ and $h$ defined in \eqref{function-R}, \eqref{function-S} and \eqref{function-L}, respectively. Thus they are all easy to solve.
However, the global convergence of ADMM \eqref{alg:ADM-3} with three blocks of variables was ambiguous. Only until very recently, was it shown that \eqref{alg:ADM-3} globally converges under certain conditions (see \cite{Luo-ADMM-2012}). It should be noted, however, that the error bound condition required in \cite{Luo-ADMM-2012} is strong and only a few classes of convex function are known that satisfy this condition.

\section{ADMM for Solving \eqref{prob:latent-variable} as a Consensus Problem}\label{sec:consensus}

Problem \eqref{prob:latent-variable} can be rewritten as a convex minimization problem with two blocks of variables and two separable functions as follows:
\be\label{prob:latent-variable-2-blocks}\ba{ll} \min & \phi(X) + \psi(Z), \\ \st & X - Z =0, \ea\ee
where $X=(R,S,L), Z=(\tilde{R},\tilde{S},\tilde{L})$, and \[\phi(X):=f(R)+g(S)+h(L), \quad \psi(Z)=\mathcal{I}(\tilde{R}-\tilde{S}+\tilde{L}=0),\] with $f,g$ and $h$ defined in \eqref{function-R}, \eqref{function-S} and \eqref{function-L}, respectively. The ADMM applied to solving \eqref{prob:latent-variable-2-blocks} can be described as follows:
\be\label{admm-2-blocks}\left\{\ba{lll} X^{k+1} & := & \argmin_X \ \phi(X) -\langle\Lambda^k,X-Z^k\rangle + \frac{1}{2\mu}\|X-Z^k\|_F^2, \\
                                       Z^{k+1} & := & \argmin_Z \ \psi(Z) -\langle\Lambda^k,X^{k+1}-Z\rangle + \frac{1}{2\mu}\|X^{k+1}-Z\|_F^2, \\
                                       \Lambda^{k+1} & := & \Lambda^k - (X^{k+1}-Z^{k+1})/\mu, \ea\right.\ee
where $\Lambda$ is the Lagrange multiplier associated with the equality constraint.
The two subproblems in \eqref{admm-2-blocks} are both easy to solve. In fact, the solution of the first subproblem in \eqref{admm-2-blocks} corresponds to the proximal mappings of $f$, $g$ and $h$. Partitioning the matrix $T^k:=X^{k+1}-\mu\Lambda^k$ into three blocks in the same form as $Z=(\tilde{R},\tilde{S},\tilde{L})$,  The second subproblem can be reduced to:
\be\label{admm-2-blocks-2nd-reduce}\ba{ll} \min & \half\|(\tilde{R},\tilde{S},\tilde{L})-(T_R^k,T_S^k,T_L^k)\|_F^2 \\ \st & \tilde{R}-\tilde{S}+\tilde{L}=0. \ea\ee The first-order optimality conditions of \eqref{admm-2-blocks-2nd-reduce} are given by
\be\label{admm-2-blocks-2nd-reduce-1st-optcond}(\tilde{R},\tilde{S},\tilde{L})-(T_R^k,T_S^k,T_L^k)-(\Gamma,-\Gamma,\Gamma)=0,\ee
where $\Gamma$ is the Lagrange multiplier associated with \eqref{admm-2-blocks-2nd-reduce}. Thus we get,
\[\tilde{R}=T_R^k+\Gamma, \quad \tilde{S}=T_S^k-\Gamma, \quad \tilde{L}=T_L^k+\Gamma.\] Substituting them into the equality constraint in \eqref{admm-2-blocks-2nd-reduce}, we get \be\label{solve-Gamma}\Gamma = -(T_R^k-T_S^k+T_L^k)/3.\ee By substituting \eqref{solve-Gamma} into \eqref{admm-2-blocks-2nd-reduce-1st-optcond} we get the solution to \eqref{admm-2-blocks-2nd-reduce}.

The ADMM \eqref{admm-2-blocks} solves Problem \eqref{prob:latent-variable-2-blocks} with two blocks of variables. It can be seen as a special case of the consensus problem discussed in \cite{Boyd-etal-ADM-survey-2011}. The global convergence result of \eqref{admm-2-blocks} has also been well studied
in the literature (see e.g., \cite{Eckstein-thesis-89,Eckstein-Bertsekas-1992}).

\section{A Proximal Gradient based Alternating Direction Method}\label{sec:ADMM}

In this section, we propose another alternating direction type method to solve \eqref{prob:latent-variable}.
In Section \ref{sec:consensus}, we managed to reduce the original problem with three blocks of variables \eqref{prob:latent-variable} to a new problem with two blocks of variables \eqref{prob:latent-variable-2-blocks}. As a result, we can use ADMM for solving problems with two blocks of variables, whose convergence has been well studied. Another way to reduce the problem \eqref{prob:latent-variable} into a problem with two blocks of variables is to group two variables (say $S$ and $L$) as one variable. This leads to the new equivalent form of \eqref{prob:latent-variable}:
\be\label{prob:latent-variable-2-blocks-new}\ba{ll}\min & f(R) + \varphi(W) \\
                                                \st & R - [I, -I]W = 0,\ea\ee
where $W=[S;L]$ and $\varphi(W)=g(S)+h(L)$. Now the ADMM for solving \eqref{prob:latent-variable-2-blocks} can be described as
\be\label{admm-2-block-new}\left\{\ba{lll} R^{k+1} & := & \argmin_R \ f(R) - \langle\Lambda^k,R-[I,-I]W^k\rangle + \frac{1}{2\mu}\|R-[I,-I]W^k\|_F^2 \\
                                           W^{k+1} & := & \argmin_W \ \varphi(W) - \langle\Lambda^k,R^{k+1}-[I,-I]W\rangle + \frac{1}{2\mu}\|R^{k+1}-[I,-I]W\|_F^2 \\ \Lambda^{k+1} & := & \Lambda^k - (R^{k+1}-[I,-I]W^{k+1})/\mu, \\\ea\right. \ee
where $\Lambda$ is the Lagrange multiplier associated with the equality constraint and $\mu>0$ is a penalty parameter. The first subproblem in \eqref{admm-2-block-new} is still easy and it corresponds to the proximal mapping of function $f$. However, the second subproblem in \eqref{admm-2-block-new} is not easy, because the two parts of $W$ are coupled together in the quadratic penalty term. To overcome this difficulty, we solve the second subproblem in \eqref{admm-2-block-new} inexactly by one step of a proximal gradient method. Note that the second subproblem in \eqref{admm-2-block-new} can be reduced to
\be\label{admm-2-block-new-sub2} W^{k+1} := \argmin_W \ \varphi(W) + \frac{1}{2\mu}\|R^{k+1}-[I,-I]W-\mu\Lambda^k\|_F^2.\ee One step of proximal gradient method solves the following problem
\be\label{admm-2-block-new-sub2-linearized} \min_W \ \varphi(W) + \frac{1}{2\mu\tau}\|W-(W^k+\tau \begin{pmatrix}I \\ -I\end{pmatrix}(R^{k+1}-[I,-I]W^k-\mu\Lambda^k))\|_F^2. \ee
Since the two parts of $W=[S;L]$ are separable in the quadratic part now, \eqref{admm-2-block-new-sub2-linearized} reduces to two problems
\be\label{admm-2-block-new-sub2-linearized-S} \min_S \ g(S) + \frac{1}{2\mu\tau}\|S-(S^k+\tau G_R^k)\|_F^2, \ee
and
\be\label{admm-2-block-new-sub2-linearized-L} \min_L \ h(L) + \frac{1}{2\mu\tau}\|L-(L^k-\tau G_R^k)\|_F^2, \ee
where $G_R^k = R^{k+1}-S^k + L^k -\mu\Lambda^k$. Both \eqref{admm-2-block-new-sub2-linearized-S} and \eqref{admm-2-block-new-sub2-linearized-L} are easy to solve as they correspond to the proximal mappings of functions $g$ and $h$, respectively. Thus, our proximal gradient based alternating direction method (PGADM) can be summarized as
\begin{algorithm}[h!]
    \caption{A Proximal Gradient based Alternating Direction Method}\label{alg:pgadm}
    {\small
    \begin{algorithmic}[1]
    \FOR{k=0,1,\ldots}
    \STATE $R^{k+1} := \argmin_R \ f(R) - \langle\Lambda^k,R-S^k+L^k\rangle + \frac{1}{2\mu}\|R-S^k+L^k\|_F^2$
    \STATE $S^{k+1} := \argmin_S \ g(S) + \frac{1}{2\mu\tau}\|S-(S^k+\tau G_R^k)\|_F^2$
    \STATE $L^{k+1} := \argmin_L \ h(L) + \frac{1}{2\mu\tau}\|L-(L^k-\tau G_R^k)\|_F^2$
    \STATE $\Lambda^{k+1} := \Lambda^k - (R^{k+1}-S^{k+1}+L^{k+1})/\mu$
    \ENDFOR
    \end{algorithmic}
    }
\end{algorithm}

\begin{remark}
The idea of incorporating proximal step into the alternating direction method of multipliers has been suggested by \cite{Eckstein-OMS-1994} and  \cite{Chen-Teboulle-1994}. This idea has then been generalized by \cite{He-Liao-Han-Yang-2002} to allow varying penalty and proximal parameters. Recently, this technique has been used for sparse and low-rank optimization problems (see \cite{Yang-Zhang-2009} and \cite{Tao-Yuan-SPCP-2011}). More recently, some convergence properties of alternating direction methods with proximal gradient steps have been studied by \cite{Deng-Yin-2012}, \cite{Luo-ADMM-2012}, \cite{Fazel-Sun-2012} and \cite{Ma-APGM-2012}. However, for the seek of completeness, we include a global convergence proof for Algorithm \ref{alg:pgadm} in the Appendix.
\end{remark}

\begin{remark}
In Algorithm \ref{alg:pgadm}, we grouped $S$ and $L$ as one block of variable. We also implemented the other two ways of grouping the variables, i.e., group $R$ and $S$ as one block, and group $R$ and $L$ as one block. We found from the numerical experiments that these two alternatives yielded similar practical performance as Algorithm \ref{alg:pgadm}.
\end{remark}

\begin{remark}
If we use the 1-off absolute penalty $\|S\|_{1,{\rm off}}:=\sum_{i \ne j}|S_{ij}|$ to replace $\|S\|_1$, our algorithm basically remains the same except that we modify
$\Shrink(\cdot,\cdot)$ as follows
\be\label{def-L1-shrinkage-modify}
[\Shrink(Z,\tau)]_{ij} := \left\{ \ba{ll} Z_{ii}, & \mbox{ if } \ i=j\\
Z_{ij} - \tau, & \mbox{ if } \ i \ne j \ \mbox{ and} \ Z_{ij} > \tau \\
                                                                    Z_{ij} + \tau, & \mbox{ if } \ i \ne j \ \mbox{ and} \ Z_{ij} < -\tau \\
                                                                    0           , & \mbox{ if } \ i \ne j \ \mbox{ and} \ -\tau \leq Z_{ij}\leq \tau. \ea\right.
\ee
\end{remark}

\section{Numerical experiments}\label{sec:numerical}

In this section, we present numerical results on both synthetic and real data to demonstrate the
efficiency of the proposed methods: ADMM \eqref{admm-2-blocks} and PGADM (Algorithm \ref{alg:pgadm}). Our codes were written in MATLAB. All numerical
experiments were run in MATLAB 7.12.0 on a laptop with Intel Core
I5 2.5 GHz CPU and 4GB of RAM.

We first compared ADMM \eqref{admm-2-blocks} with PGADM (Algorithm \ref{alg:pgadm}) on some synthetic problems.
We compared ADMM and PGADM using two different ways of choosing $\mu$. One set of comparisons used a fixed $\mu=10$, and the other set of comparisons used a continuation scheme to dynamically change $\mu$. The continuation scheme we used was to set the initial value of $\mu$ as the size of the matrix $p$, and then multiply $\mu$ by $1/4$ after every 10 iterations.

We then compared the performance of PGADM (with continuation on $\mu$) with LogdetPPA proposed by \cite{Wang-Sun-Toh-2009} and used in \cite{Chandrasekaran-Parrilo-Willsky-2010} for solving \eqref{prob:latent-variable}.

\subsection{Comparison of ADMM and PGADM on Synthetic Data}

We observed form the numerical experiments that the step size $\tau$ of the proximal gradient step in PGADM (Algorithm \ref{alg:pgadm}) can be slightly larger than $1/2$ and the algorithm produced very good results. We thus chose the step size $\tau$ to be $0.6$ in our experiments.

We randomly created test problems using a procedure proposed by \cite{Scheinberg-Rish-2009} and \cite{Scheinberg-Ma-Goldfarb-NIPS-2010} for the classical graphical lasso problems. Similar procedures were used by \cite{Wang-Sun-Toh-2009} and \cite{Li-Toh-2010}.
For a given number of observed variables $p$ and a given number of latent variables $p_h$, we first created a sparse matrix $U\in \br^{(p+p_h)\times (p+p_h)}$ with sparsity around 10\%, i.e., 10\% of the entries are nonzeros. The nonzero entries were set to -1 or 1 with equal probability. Then we computed $K:=(U*U^\top)^{-1}$ as the true covariance matrix. We then chose the submatrix of $K$, $\hat{S}:=K(1:p,1:p)$ as the ground truth matrix of the sparse matrix $S$ and chose $\hat{L}:=K(1:p,p+1:p+p_h)K(p+1:p+p_h,p+1:p+p_h)^{-1}K(p+1:p+p_h,1:p)$ as the ground truth matrix of the low rank matrix $L$.  We then drew $N=5p$ iid vectors, $Y_1,\ldots,Y_N$, from the Gaussian distribution $\mathcal{N}(\mathbf{0}, (\hat{S}-\hat{L})^{-1})$ by using the $mvnrnd$ function in MATLAB, and computed a sample covariance matrix of the observed variables $\Sigma_X:=\frac{1}{N}\sum_{i=1}^N Y_iY_i^\top.$

We computed the relative infeasibility of the sequence $(R^k,S^k,L^k)$ generated by inexact ADMM using
\be \label{infeas-ADMM} \mbox{infeas} := \frac{\|R^k-S^k+L^k\|_F}{\max\{1,\|R^k\|_F,\|S^k\|_F,\|L^k\|_F\}}. \ee

In the comparison of ADMM and PGADM, the size of all problems was chosen as $p=1000$. For fixed $\mu=10$, we first ran the ADMM for 100 iterations, and recorded the objective function value and $infeas$. We then ran PGADM until it achieves an objective function value within relative error $10^{-5}$ compared with the objective function value given by ADMM, or it achieves an $infeas$ within relative error $10^{-5}$ compared with the $infeas$ given by ADMM. The number of iterations, CPU times, $infeas$ and objective function values for both ADMM and PGADM were reported in Table \ref{tab:synthetic-boyd-1000-fixed-mu}.
From Table \ref{tab:synthetic-boyd-1000-fixed-mu}, we see that for fixed $\mu=10$, ADMM was faster than PGADM when $\alpha$ and $\beta$ are both small, and PGADM was faster than ADMM when $\alpha$ and $\beta$ are both large.

\begin{table}[ht]{
\begin{center}\caption{Comparison of ADMM with PGADM (for fixed $\mu$) on synthetic data}\label{tab:synthetic-boyd-1000-fixed-mu}
\vspace{0.1in}
\begin{tabular}{|cc||c|c|c|c||c|c|c|c|}\hline
\multicolumn{2}{|c|}{} & \multicolumn{4}{|c||}{PGADM} & \multicolumn{4}{|c|}{ADMM} \\\hline

$\alpha$ & $\beta$ & obj           & iter  & cpu   & infeas    & obj          & iter & cpu    & infeas    \\\hline

0.005    & 0.025     & -1.6987e+002  & 135   & 224.1 & 7.6e-005 & -1.7098e+002 & 100  & 172.9  & 6.6e-005 \\\hline

0.005    & 0.05      & -9.3385e+001  & 146   & 243.7 & 3.8e-005 & -9.3476e+001 & 100  & 173.7  & 2.9e-005 \\\hline

0.01     & 0.05      & -4.4748e+001  & 132   & 214.7 & 1.7e-005 & -4.4761e+001 & 100  & 180.1  & 7.0e-006 \\\hline

0.01     & 0.1       & 5.4571e+001   & 111   & 177.4 & 9.8e-006 & 5.4567e+001  & 100  & 166.9  & 2.8e-007  \\\hline

0.02     & 0.1       & 1.1881e+002   & 83    & 136.8 & 9.7e-006 & 1.1881e+002  & 100  & 170.1  & 1.9e-007  \\\hline

0.02     & 0.2       & 2.3717e+002   & 56    & 91.5  & 9.3e-006 & 2.3717e+002  & 100  & 174.0  & 1.6e-007  \\\hline

0.04     & 0.2       & 3.2417e+002   & 44    & 67.0  & 9.2e-006 & 3.2417e+002  & 100  & 165.2  & 1.6e-007  \\\hline

0.04     & 0.4       & 4.5701e+002   & 19    & 29.4  & 3.1e-004 & 4.5700e+002  & 100  & 168.4  & 1.7e-007    \\\hline

\end{tabular}
\end{center}}
\end{table}

We then further compare ADMM and PGADM on synthetic data with the continuation scheme for $\mu$ discussed above. We terminated both ADMM and PGADM when $infeas<10^{-5}$. We reported the results in Table \ref{tab:synthetic-boyd-1000-continuation-mu}.

From Table \ref{tab:synthetic-boyd-1000-continuation-mu}, we see that the continuation scheme used really helped to speed up the convergence and produced much better results. Also, using this continuation scheme, PGADM was faster than ADMM with comparable residuals and objective function values. However, we should remark that PGADM was faster than ADMM using the specific continuation scheme. If other continuation schemes were adopted, the results could be quite different. In the comparison with LogdetPPA in the following sections, we only compare LogdetPPA with PGADM with this continuation scheme.

\begin{table}[ht]{
\begin{center}\caption{Comparison of ADMM with PGADM (with continuation for $\mu$) on synthetic data}\label{tab:synthetic-boyd-1000-continuation-mu}
\vspace{0.1in}
\begin{tabular}{|cc||c|c|c|c||c|c|c|c|}\hline
\multicolumn{2}{|c|}{} & \multicolumn{4}{|c||}{PGADM} & \multicolumn{4}{|c|}{ADMM} \\\hline

$\alpha$ & $\beta$ & obj           & iter  & cpu   & resid    & obj          & iter & cpu    & resid    \\\hline


0.005 & 0.025 & -1.711329e+002 & 32 & 51.5 & 5.7e-006 & -1.711334e+002 & 61 & 103.7 & 6.7e-006     \\\hline

0.005 & 0.05  & -9.348245e+001 & 41 & 65.8 & 4.2e-006 & -9.348589e+001 & 62 & 103.2 & 6.6e-006     \\\hline

0.010 & 0.05  & -4.476323e+001 & 41 & 67.8 & 2.8e-006 & -4.476499e+001 & 62 & 105.7 & 9.8e-006     \\\hline

0.010 & 0.10  & 5.456790e+001 & 41 & 65.7 & 5.1e-006 & 5.456724e+001 & 71 & 119.6 & 6.0e-006     \\\hline

0.020 & 0.10  & 1.188077e+002 & 41 & 64.9 & 3.4e-006 & 1.188054e+002 & 71 & 115.1 & 8.1e-006     \\\hline

0.020 & 0.20  & 2.371659e+002 & 45 & 76.4 & 8.7e-006 & 2.371687e+002 & 75 & 125.6 & 7.0e-006     \\\hline

0.040 & 0.20  & 3.241688e+002 & 44 & 72.7 & 8.3e-006 & 3.241684e+002 & 75 & 126.6 & 8.5e-006     \\\hline

0.040 & 0.40  & 4.570019e+002 & 50 & 77.5 & 7.5e-006 & 4.570058e+002 & 78 & 126.5 & 9.7e-006     \\\hline

\end{tabular}
\end{center}}
\end{table}

\subsection{Comparison of PGADM and LogdetPPA on Synthetic Data}

In this section, we compare PGADM with LogdetPPA on synthetic data created the same way as in the last section.
LogdetPPA, proposed by Wang \etal in \cite{Wang-Sun-Toh-2009}, is a proximal point algorithm for solving semidefinite programming problems with $\log\det(\cdot)$ function. The specialized MATLAB codes of LogdetPPA for solving \eqref{prob:latent-variable} were downloaded from http://ssg.mit.edu/$\sim$venkatc/latent-variable-code.html.

We compared PGADM (with continuation on $\mu$) with LogdetPPA with different $\alpha$ and $\beta$. We reported the comparison results on objective function value, CPU time, sparsity of $S$ and infeas in Table \ref{tab:synthetic}. The sparsity of $S$ is denoted as
\[sp := \frac{\#\{(i,j):S_{ij}\neq 0\}}{p^2},\] i.e., the percentage of nonzero entries. Since matrix $S$ generated by LogdetPPA is always dense but with many small entries, we also measure its sparsity by truncating small entries that less than $10^{-4}$ to zeros, i.e.,
\[sp1 := \frac{\#\{(i,j):|S_{ij}| > 10^{-4}\}}{p^2}.\]
All CPU times reported are in seconds. We report the speed up of PGADM over LogdetPPA in Table \ref{tab:synthetic:speedup}.

\begin{table}[ht]{
\begin{center}\caption{Results of PGADM and LogdetPPA on synthetic data}\label{tab:synthetic}
\vspace{0.1in}
\begin{tabular}{|c||c|r|c|c||c|r|c|c|}\hline
dim & \multicolumn{4}{|c||}{LogdetPPA} & \multicolumn{4}{|c|}{PGADM} \\\hline
p   & obj & cpu & sp (\%) & sp1 (\%)           & obj & cpu & sp (\%) & infeas \\\hline
\multicolumn{9}{|c|}{$\alpha=0.005$, $\beta=0.025$}               \\\hline
200 & 1.914315e+2 & 7.2 & 100.00 & 19.17 & 1.910379e+2 & 1.0 & 18.90 & 4.3e-6 \\\hline
500 & 1.898418e+2 & 235.2 & 100.00 & 5.78 & 1.898275e+2 & 7.3 & 5.63 & 7.2e-6 \\\hline
1000 & -1.711293e+2 & 1706.0 & 100.00 & 0.52 & -1.711329e+2 & 48.2 & 0.49 & 5.7e-6  \\\hline
2000 & -1.430010e+3 & 5001.2 & 100.00 & 0.06 & -1.435605e+3 & 427.2 & 0.05 & 4.9e-6 \\\hline
\multicolumn{9}{|c|}{$\alpha=0.005$, $\beta=0.05$}               \\\hline
200 & 1.926376e+2 & 29.1 & 100.00 & 43.63 & 1.924829e+2 & 2.5 & 48.66 & 6.6e-6 \\\hline
500 & 2.051884e+2 & 358.3 & 100.00 & 12.04 & 2.051425e+2 & 10.5 & 11.42 & 8.8e-6  \\\hline
1000 & -9.347297e+1 & 1076.4 & 100.00 & 4.88 & -9.348245e+1 & 71.6 & 4.72 & 4.2e-6  \\\hline
2000 & -1.229323e+3 & 5191.5 & 100.00 & 0.21 & -1.230238e+3 & 445.0 & 0.15 & 9.5e-6  \\\hline
\multicolumn{9}{|c|}{$\alpha=0.01$, $\beta=0.05$}               \\\hline
200 & 2.030390e+2 & 20.8 & 100.00 & 20.46 & 2.026586e+2 & 2.6 & 11.06 & 9.0e-6 \\\hline
500 & 2.394720e+2 & 146.0 & 100.00 & 4.25 & 2.394631e+2 & 10.7 & 4.14 & 3.7e-6 \\\hline
1000 & -4.476078e+1 & 740.6 & 100.00 & 0.24 & -4.476323e+1 & 72.8 & 0.23 & 2.8e-6  \\\hline
2000 & -1.101454e+3 & 6433.4 & 100.00 & 0.05 & -1.111504e+3 & 453.7 & 0.05 & 6.4e-6   \\\hline
\multicolumn{9}{|c|}{$\alpha=0.01$, $\beta=0.1$}               \\\hline
200 & 2.050879e+2 & 29.1 & 100.00 & 42.16 & 2.048359e+2 & 1.6 & 35.83 & 7.0e-6 \\\hline
500 & 2.639548e+2 & 235.2 & 100.00 & 8.62 & 2.638565e+2 & 11.6 & 8.19 & 5.8e-6  \\\hline
1000 & 5.456825e+1 & 932.1 & 100.00 & 2.26 & 5.456790e+1 & 74.9 & 2.26 & 5.1e-6   \\\hline
2000 & -8.541802e+2 & 4813.6 & 100.00 & 0.14 & -8.712916e+2 & 516.6 & 0.07 & 8.6e-6 \\\hline
\end{tabular}
\end{center}}
\end{table}

From Table \ref{tab:synthetic} we see that the solutions produced by PGADM always have comparable objective function values compared to the solutions produced by LogdetPPA. However, our PGADM is always much faster than LogdetPPA, as shown in both Tables \ref{tab:synthetic} and \ref{tab:synthetic:speedup}. In fact, PGADM is usually ten times faster than LogdetPPA, and sometimes more than thirty five times faster. For example, for the four large problems with matrices size $2000\times 2000$, LogdetPPA needs 1 hour 23 minutes, 1 hour 26 minutes, 1 hour 47 minutes and 1 hour 20 minutes, respectively, to solve them, while our PGADM needs about 7 minutes, 7 minutes, 8 minutes and 9 minutes respectively to solve them. We also notice that the matrix $S$ generated by PGADM is always a sparse matrix with many entries exactly equal to zero, but $S$ generated by LogdetPPA is always a dense matrix, and only when we truncate the entries that are smaller than $10^{-4}$ to zeros, it becomes a sparse matrix with similar level of sparsity. This is because in our PGADM, $S$ is updated by the $\ell_1$ shrinkage operation, which truncates the small entries to zeros, while LogdetPPA needs to replace $\|S\|_1$ with smooth linear function which does not preserve sparsity.

\begin{table}[ht]{
\begin{center}\caption{Speed up of PGADM over LogdetPPA on synthetic data.}\label{tab:synthetic:speedup}
\vspace{0.1in}
\begin{tabular}{|c||c|c|c|c|}\hline
p   & $\alpha=0.005$, $\beta=0.025$ & $\alpha=0.005$, $\beta=0.05$ & $\alpha=0.01$, $\beta=0.05$ & $\alpha=0.01$, $\beta=0.1$ \\\hline
200 &  7.1  & 11.7  & 7.9 & 18.7 \\\hline
500 & 32.0  & 34.2  & 13.6 & 20.3 \\\hline
1000 & 35.4 & 15.0 & 10.2 & 12.4 \\\hline
2000 & 11.7 & 11.7 & 14.2  & 9.3   \\\hline
\end{tabular}
\end{center}}
\end{table}

\subsection{Comparison of PGADM and LogdetPPA on Gene expression data}

To further demonstrate the efficacy of PGADM, we applied PGADM to solving \eqref{prob:latent-variable} with two gene expression data sets. One data set is the Rosetta Inpharmatics Compendium of gene expression data (denoted as Rosetta) \cite{Hughes-Rosetta-2000} profiles
which contains $301$ samples with $6316$ variables (genes).
The other data set is the Iconix microarray data set (denoted as Iconix) from drug
treated rat livers \cite{Natsoulis-GSE8858-2005} which contains $255$ samples with $10455$ variables.

For a given number of observed variables $p$, we created the sample covariance matrix $\Sigma_X$ by the following procedure. We first computed the variances of all of variables using all the sample data. We then selected the $p$ variables with the highest variances and computed the sample covariance matrix $\Sigma_X$ of these $p$ variables using all the sample data. We reported the comparison results of PGADM and LogdetPPA in Tables \ref{tab:Rosetta} and \ref{tab:Iconix} for the Rosetta and Iconix data sets, respectively. Table \ref{tab:real:speedup} summarizes the speed up of PGADM over LogdetPPA.

\begin{table}[ht]{
\begin{center}\caption{Results of PGADM and LogdetPPA on Rosetta data set}\label{tab:Rosetta}
\vspace{0.1in}
\begin{tabular}{|c||c|r|c|c||c|r|c|c|}\hline
dim & \multicolumn{4}{|c||}{LogdetPPA} & \multicolumn{4}{|c|}{PGADM} \\\hline
p   & obj & cpu & sp (\%) & sp1 (\%)           & obj & cpu & sp (\%) & infeas \\\hline
200 & -2.726188e+2 & 5.6 & 100.00 & 0.58 & -2.726184e+2 & 0.7 & 0.58 & 7.1e-6  \\\hline
500 & -8.662116e+2 & 68.0 & 100.00 & 0.20 & -8.662113e+2 & 7.9 & 0.20 & 3.7e-6  \\\hline
1000 & -1.970974e+3 & 490.9 & 100.00 & 0.10 & -1.970973e+3 & 52.2 & 0.10 & 7.7e-6  \\\hline
2000 & -4.288406e+3 & 4597.9 & 100.00 & 0.05 & -4.288406e+3 & 422.3 & 0.05 & 5.4e-6 \\\hline
\end{tabular}
\end{center}}
\end{table}

\begin{table}[ht]{
\begin{center}\caption{Results of PGADM and LogdetPPA on Iconix data set}\label{tab:Iconix}
\vspace{0.1in}
\begin{tabular}{|c||c|r|c|c||c|r|c|c|}\hline
dim & \multicolumn{4}{|c||}{LogdetPPA} & \multicolumn{4}{|c|}{PGADM}  \\\hline
p   & obj & cpu & sp (\%) & sp1  (\%)          & obj & cpu & sp (\%) & infeas  \\\hline
200 & 1.232884e+3 & 38.5 & 100.00 & 36.10 & 1.232744e+3 & 2.5 & 41.62 & 7.2e-6  \\\hline
500 & 1.842623e+3 & 98.8 & 100.00 & 2.27 & 1.839838e+3 & 17.0 & 0.77 & 1.0e-5  \\\hline
1000 & 1.439052e+3 & 1341.9 & 100.00 & 1.45 & 1.435425e+3 & 94.6 & 0.14 & 6.3e-6  \\\hline
2000 & 1.242966e+2 & 13207.2 & 100.00 & 0.07 & 1.168757e+2 & 738.2 & 0.06 & 8.5e-6 \\\hline
\end{tabular}
\end{center}}
\end{table}

From Table \ref{tab:Rosetta}, we again see that PGADM always generates solutions with comparable objective function values in much less time. For example, for $p=2000$, LogdetPPA needs 1 hour 16 minutes to solve it while PGADM takes just 7 minutes. From Table \ref{tab:Iconix}, we see that the advantage of PGADM is more obvious. For $p=200,500,1000$ and $2000$, PGADM always generates solutions with much smaller objective function values and it is always much faster than LogdetPPA. For example, for $p=2000$, LogdetPPA takes 3 hours 40 minutes to solve it while PGADM just takes about 12 minutes.

\begin{table}[ht]{
\begin{center}\caption{Speed up of PGADM over LogdetPPA on Rosetta and Iconix data sets.}\label{tab:real:speedup}
\vspace{0.1in}
\begin{tabular}{|c||c|c|}\hline
dim & Rosetta data & Iconix data\\\hline
200 &  8.0  & 15.4\\\hline
500 & 8.6& 5.8  \\\hline
1000 &9.4 & 14.2\\\hline
2000 & 10.9 & 17.9 \\\hline
\end{tabular}
\end{center}}
\end{table}

\section{Conclusion}\label{sec:conclusion}

In this paper, we proposed alternating direction methods for solving latent variable Gaussian graphical model selection. The global convergence results of our methods were established. We applied the proposed methods for solving large problems from both synthetic data and gene expression data. The numerical results indicated that our methods were five to thirty five times faster than a state-of-the-art Newton-CG proximal point algorithm.

\section*{Acknowledgement} The authors thank Professor Stephen Boyd for suggesting solving \eqref{prob:latent-variable} as a consensus problem \eqref{prob:latent-variable-2-blocks} using ADMM. Shiqian Ma's research is supported by the National Science Foundation postdoctoral fellowship through the Institute for Mathematics and Its Applications at University of Minnesota. Hui Zou's research is supported in part by grants from the National Science Foundation and the Office of Naval Research.

\bibliographystyle{siam}
\bibliography{C:/Mywork/Optimization/work/reports/bibfiles/All,precision}

\begin{thebibliography}{10}

\bibitem{ahmed2009}
{\sc A.~Ahmed and E.P. Xing}, {\em Recovering time-varying networks of
  dependencies in social and biological studies}, Proceedings of the National
  Academy of Sciences, 106 (2009), pp.~11878--11883.

\bibitem{banerjee2008}
{\sc O.~Banerjee, L.~El~Ghaoui, and A.~d'Aspremont}, {\em {Model selection
  through sparse maximum likelihood estimation for multivariate Gaussian or
  binary data}}, Journal of Machine Learning Research, 9 (2008), pp.~485--516.

\bibitem{Boyd-etal-ADM-survey-2011}
{\sc S.~Boyd, N.~Parikh, E.~Chu, B.~Peleato, and J.~Eckstein}, {\em Distributed
  optimization and statistical learning via the alternating direction method of
  multipliers}, Foundations and Trends in Machine Learning,  (2011).

\bibitem{BV-ConvexBook2004}
{\sc S.~Boyd and L.~Vandenberghe}, {\em Convex optimization}, Cambridge
  University Press, Cambridge, 2004.

\bibitem{clime11}
{\sc T.~Cai, W.~Liu, and X.~Luo}, {\em {A Constrained $\ell_1$ Minimization
  Approach to Sparse Precision Matrix Estimation}}, Journal of the American
  Statistical Association, 106 (2011), pp.~594--607.

\bibitem{candes2007}
{\sc E.~Cand\`es and T.~Tao}, {\em {The Dantzig selector: Statistical
  estimation when p is much larger than n}}, The Annals of Statistics, 35
  (2007), pp.~2313--2351.

\bibitem{Chandrasekaran-Parrilo-Willsky-2010}
{\sc V.~Chandrasekaran, P.A. Parrilo, and A.S. Willsky}, {\em Latent variable
  graphical model selection via convex optimization}, preprint,  (2010).

\bibitem{Chen-Teboulle-1994}
{\sc G.~Chen and M.~Teboulle}, {\em A proximal-based decomposition method for
  convex minimization problems}, Mathematical Programming, 64 (1994),
  pp.~81--101.

\bibitem{Combettes-Pesquet-DR-2007}
{\sc P.~L. Combettes and Jean-Christophe Pesquet}, {\em A {D}ouglas-{R}achford
  splitting approach to nonsmooth convex variational signal recovery}, IEEE
  Journal of Selected Topics in Signal Processing, 1 (2007), pp.~564--574.

\bibitem{Combettes-Wajs-05}
{\sc P.~L. Combettes and V.~R. Wajs}, {\em Signal recovery by proximal
  forward-backward splitting}, SIAM Journal on Multiscale Modeling and
  Simulation, 4 (2005), pp.~1168--1200.

\bibitem{Aspremont-Banerjee-ElGhaoui-2008}
{\sc A.~D'Aspremont, O.~Banerjee, and L.~El~Ghaoui}, {\em First-order methods
  for sparse covariance selection}, SIAM Journal on Matrix Analysis and its
  Applications, 30 (2008), pp.~56--66.

\bibitem{Deng-Yin-2012}
{\sc W.~Deng and W.~Yin}, {\em On the global and linear convergence of the
  generalized alternating direction method of multipliers}, tech. report, Rice
  University CAAM, Technical Report TR12-14, 2012.

\bibitem{dobra2009}
{\sc A.~Dobra, T.S. Eicher, and A.~Lenkoski}, {\em {Modeling uncertainty in
  macroeconomic growth determinants using Gaussian graphical models}},
  Statistical Methodology, 7 (2009), pp.~292--306.

\bibitem{Douglas-Rachford-56}
{\sc J.~Douglas and H.~H. Rachford}, {\em On the numerical solution of the heat
  conduction problem in 2 and 3 space variables}, Transactions of the American
  Mathematical Society, 82 (1956), pp.~421--439.

\bibitem{Duchi-UAI-2008}
{\sc J.~Duchi, S.~Gould, and D.~Koller}, {\em Projected subgradient methods for
  learning sparse {G}aussian}, Conference on Uncertainty in Artificial
  Intelligence (UAI 2008),  (2008).

\bibitem{Eckstein-thesis-89}
{\sc J.~Eckstein}, {\em Splitting methods for monotone operators with
  applications to parallel optimization}, PhD thesis, Massachusetts Institute
  of Technology, 1989.

\bibitem{Eckstein-OMS-1994}
\leavevmode\vrule height 2pt depth -1.6pt width 23pt, {\em Some saddle-function
  splitting methods for convex programming}, Optimization Methods and Software,
  4 (1994), pp.~75--83.

\bibitem{Eckstein-Bertsekas-1992}
{\sc J.~Eckstein and D.~P. Bertsekas}, {\em On the {D}ouglas-{R}achford
  splitting method and the proximal point algorithm for maximal monotone
  operators}, Math. Program., 55 (1992), pp.~293--318.

\bibitem{edwards2000}
{\sc D.~Edwards}, {\em {Introduction to graphical modelling}}, Springer Verlag,
  2000.

\bibitem{Fazel-Sun-2012}
{\sc M.~Fazel, T.~Pong, D.~Sun, and P.~Tseng}, {\em Hankel matrix rank
  minimization with applications to system identification and realization},
  preprint,  (2012).

\bibitem{friedman2008}
{\sc J.~Friedman, T.~Hastie, and R.~Tibshirani}, {\em {Sparse inverse
  covariance estimation with the graphical lasso}}, Biostatistics, 9 (2008),
  pp.~432--441.

\bibitem{nfriedman2004}
{\sc N.~Friedman}, {\em {Inferring cellular networks using probabilistic
  graphical models}}, Science, 303 (2004), p.~799.

\bibitem{Gabay-83}
{\sc D.~Gabay}, {\em Applications of the method of multipliers to variational
  inequalities}, in Augmented Lagrangian Methods: Applications to the Solution
  of Boundary Value Problems, M.~Fortin and R.~Glowinski, eds., North-Hollan,
  Amsterdam, 1983.

\bibitem{Glowinski-LeTallec-89}
{\sc R.~Glowinski and P.~Le~Tallec}, {\em Augmented Lagrangian and
  Operator-Splitting Methods in Nonlinear Mechanics}, SIAM, Philadelphia,
  Pennsylvania, 1989.

\bibitem{Goldfarb-Ma-Ksplit}
{\sc D.~Goldfarb and S.~Ma}, {\em Fast multiple splitting algorithms for convex
  optimization}, SIAM Journal on Optimization, 22 (2012), pp.~533--556.

\bibitem{Goldfarb-Ma-Scheinberg-2010}
{\sc D.~Goldfarb, S.~Ma, and K.~Scheinberg}, {\em Fast alternating
  linearization methods for minimizing the sum of two convex functions},
  Mathematical Programming Series A, to appear,  (2012).

\bibitem{Goldstein-Osher-08}
{\sc T.~Goldstein and S.~Osher}, {\em The split {Bregman} method for
  {L1}-regularized problems}, SIAM J. Imaging Sci., 2 (2009), pp.~323--343.

\bibitem{He-Liao-Han-Yang-2002}
{\sc B.~S. He, L.-Z. Liao, D.~Han, and H.~Yang}, {\em A new inexact alternating
  direction method for monotone variational inequalities}, Math. Program., 92
  (2002), pp.~103--118.

\bibitem{Hughes-Rosetta-2000}
{\sc T.~R. Hughes, M.~J. Marton, A.~R. Jones, C.~J. Roberts, R.~Stoughton,
  C.~D. Armour, H.~A. Bennett, E.~Coffey, H.~Dai, Y.~D. He, M.~J. Kidd, A.~M.
  King, M.~R. Meyer, D.~Slade, P.~Y. Lum, S.~B. Stepaniants, D.~D. Shoemaker,
  D.~Gachotte, K.~Chakraburtty, J.~Simon, M.~Bard, and S.~H. Friend}, {\em
  Functional discovery via a compendium of expression profiles}, Cell, 102
  (2000), p.~109–126.

\bibitem{kolar2010}
{\sc M.~Kolar, L.~Song, A.~Ahmed, and E.P. Xing}, {\em Estimating time-varying
  networks}, The Annals of Applied Statistics, 4 (2010), pp.~94--123.

\bibitem{lauritzen1996}
{\sc S.L. Lauritzen}, {\em {Graphical models}}, Oxford University Press, USA,
  1996.

\bibitem{Li-Toh-2010}
{\sc L.~Li and K.-C. Toh}, {\em An inexact interior point method for
  $l_1$-regularized sparse covariance selection}, preprint,  (2010).

\bibitem{lisz2009}
{\sc S.Z. Li}, {\em {Markov random field modeling in image analysis}},
  Springer-Verlag New York Inc, 2009.

\bibitem{Lions-Mercier-79}
{\sc P.~L. Lions and B.~Mercier}, {\em Splitting algorithms for the sum of two
  nonlinear operators}, SIAM Journal on Numerical Analysis, 16 (1979),
  pp.~964--979.

\bibitem{Lu-covsel-siopt-2009}
{\sc Z.~Lu}, {\em Smooth optimization approach for sparse covariance
  selection}, SIAM J. Optim., 19 (2009), pp.~1807--1827.

\bibitem{Luo-ADMM-2012}
{\sc Z.~Q. Luo}, {\em On the linear convergence of alternating direction method
  of multipliers}, preprint,  (2012).

\bibitem{Ma-SPCA-2011-submit}
{\sc S.~Ma}, {\em Alternating direction method of multipliers for sparse
  principal component analysis}, preprint,  (2011).

\bibitem{Ma-APGM-2012}
\leavevmode\vrule height 2pt depth -1.6pt width 23pt, {\em Alternating proximal
  gradient method for convex minimization}, preprint,  (2012).

\bibitem{Malick-Povh-Rendl-Wiegele-2009}
{\sc J.~Malick, J.~Povh, F.~Rendl, and A.~Wiegele}, {\em Regularization methods
  for semidefinite programming}, SIAM Journal on Optimization, 20 (2009),
  pp.~336--356.

\bibitem{meinshausen2006}
{\sc N.~Meinshausen and P.~B{\"u}hlmann}, {\em {High-dimensional graphs and
  variable selection with the lasso}}, The Annals of Statistics,  (2006),
  pp.~1436--1462.

\bibitem{Natsoulis-GSE8858-2005}
{\sc G.~Natsoulis, L.~El Ghaoui, G.~Lanckriet, A.~Tolley, F.~Leroy, S.~Dunlea,
  B.~Eynon, C.~Pearson, S.~Tugendreich, and K.~Jarnagin}, {\em Classification
  of a large microarray data set: algorithm comparison and analysis of drug
  signatures}, Genome Research, 15 (2005), p.~724 –736.

\bibitem{Boyd-NIPS-2011}
{\sc N.~Parikh and S.~Boyd}, {\em Block splitting for large-scale distributed
  learning}, in NIPS, 2011.

\bibitem{Peaceman-Rachford-55}
{\sc D.~H. Peaceman and H.~H. Rachford}, {\em The numerical solution of
  parabolic elliptic differential equations}, SIAM Journal on Applied
  Mathematics, 3 (1955), pp.~28--41.

\bibitem{peng2009}
{\sc J.~Peng, P.~Wang, N.~Zhou, and J.~Zhu}, {\em Partial correlation
  estimation by joint sparse regression models}, Journal of the American
  Statistical Association, 104 (2009), pp.~735--746.

\bibitem{Qin-Goldfarb-Ma-2011}
{\sc Z.~Qin, D.~Goldfarb, and S.~Ma}, {\em An alternating direction method for
  total variation denoising}, preprint,  (2011).

\bibitem{ravikumar2008}
{\sc P.~Ravikumar, M.J. Wainwright, G.~Raskutti, and B.~Yu}, {\em
  {High-dimensional covariance estimation by minimizing $\ell_1$-penalized
  log-determinant divergence}}, Advances in Neural Information Processing
  Systems,  (2008).

\bibitem{rothman2008}
{\sc A.J. Rothman, P.J. Bickel, E.~Levina, and J.~Zhu}, {\em {Sparse
  permutation invariant covariance estimation}}, Electronic Journal of
  Statistics, 2 (2008), pp.~494--515.

\bibitem{Scheinberg-Ma-Goldfarb-NIPS-2010}
{\sc K.~Scheinberg, S.~Ma, and D.~Goldfarb}, {\em Sparse inverse covariance
  selection via alternating linearization methods}, in Proceedings of the
  Neural Information Processing Systems (NIPS), 2010.

\bibitem{Scheinberg-Rish-2009}
{\sc K.~Scheinberg and I.~Rish}, {\em Sinco - a greedy coordinate ascent method
  for sparse inverse covariance selection problem},  (2009).
\newblock Preprint available at
  http://www.optimization-online.org/DB\_HTML/2009/07/2359.html.

\bibitem{Tao-Yuan-SPCP-2011}
{\sc M.~Tao and X.~Yuan}, {\em Recovering low-rank and sparse components of
  matrices from incomplete and noisy observations}, SIAM J. Optim., 21 (2011),
  pp.~57--81.

\bibitem{tibshirani1996}
{\sc R.~Tibshirani}, {\em {Regression shrinkage and selection via the lasso}},
  Journal of the Royal Statistical Society. Series B, 58 (1996), pp.~267--288.

\bibitem{Wang-Sun-Toh-2009}
{\sc C.~Wang, D.~Sun, and K.-C. Toh}, {\em Solving log-determinant optimization
  problems by a {N}ewton-{CG} primal proximal point algorithm}, preprint,
  (2009).

\bibitem{Wang-Yang-Yin-Zhang-2008}
{\sc Y.~Wang, J.~Yang, W.~Yin, and Y.~Zhang}, {\em A new alternating
  minimization algorithm for total variation image reconstruction}, SIAM
  Journal on Imaging Sciences, 1 (2008), pp.~248--272.

\bibitem{Wen-Goldfarb-Yin-2009}
{\sc Z.~Wen, D.~Goldfarb, and W.~Yin}, {\em Alternating direction augmented
  {L}agrangian methods for semidefinite programming}, Mathematical Programming
  Computation, 2 (2010), pp.~203--230.

\bibitem{wille2004}
{\sc A.~Wille, P.~Zimmermann, E.~Vranov{\'a}, A.~F{\"u}rholz, O.~Laule,
  S.~Bleuler, L.~Hennig, A.~Prelic, P.~Von~Rohr, L.~Thiele, et~al.}, {\em
  {Sparse graphical Gaussian modeling of the isoprenoid gene network in
  Arabidopsis thaliana}}, Genome Biology, 5 (2004), p.~R92.

\bibitem{Witten11}
{\sc D.M. Witten, J.H. Friedman, and N.~Simon}, {\em {New insights and faster
  computations for the graphical lasso}}, Journal of Computational and
  Graphical Statistics, 20 (2011), pp.~892--900.

\bibitem{Yang-Zhang-2009}
{\sc J.~Yang and Y.~Zhang}, {\em Alternating direction algorithms for $\ell_1$
  problems in compressive sensing}, SIAM Journal on Scientific Computing, 33
  (2011), pp.~250--278.

\bibitem{yuan2010}
{\sc M.~Yuan}, {\em {High Dimensional Inverse Covariance Matrix Estimation via
  Linear Programming}}, The Journal of Machine Learning Research, 11 (2010),
  pp.~2261--2286.

\bibitem{yuan2007}
{\sc M.~Yuan and Y.~Lin}, {\em {Model selection and estimation in the Gaussian
  graphical model}}, Biometrika, 94 (2007), pp.~19--35.

\bibitem{Yuan-2009}
{\sc X.~Yuan}, {\em Alternating direction methods for sparse covariance
  selection}, Journal of Scientific Computing,  (2009).

\end{thebibliography}

\appendix
\section{Global Convergence Analysis of PGADM}\label{sec:convergence}

In this section, we establish the global convergence result of PGADM (Algorithm \ref{alg:pgadm}). This convergence proof is not much different with the one given by \cite{Yang-Zhang-2009} for compressed sensing problems. We include the proof here just for completeness.

We introduce some notation first. We define $W=\begin{pmatrix} S\\ L\end{pmatrix}$. We define functions $F(\cdot)$ and $G(\cdot)$ as
\[ F(R) := \langle R,\hat\Sigma_X\rangle - \log\det R, \]
and
\[ G(W) := \alpha\|S\|_1 + \beta \Tr(L) + \mathcal{I}(L\succeq 0).\]
Note that both $F$ and $G$ are convex functions.
We also define matrix $A$ as $A=[-I_{p\times p}, I_{p\times p}]\in\br^{p\times 2p}$. Now Problem \eqref{prob:latent-variable} can be rewritten as
\be \label{prob:generic-min-sum-2} \min\  F(R) + G(W), \quad \st \ R + AW = 0, \ee and our PGADM (Algorithm \ref{alg:pgadm}) can be rewritten as
\be \label{alg:ADM-generic}\left\{\ba{lll} R^{k+1} & := & \argmin_R \ F(R) + G(W^k) - \langle \Lambda^k, R+AW^k \rangle + \frac{1}{2\mu}\|R+AW^k\|_F^2 \\
                                    W^{k+1} & := & \argmin_W \ F(R^{k+1}) + G(W) + \frac{1}{2\tau\mu}\|W-\left(W^k-\tau A^\top(R^{k+1}+AW^k-\mu\Lambda^k)\right)\|_F^2 \\
                                    \Lambda^{k+1} & := & \Lambda^k - (R^{k+1}+AW^{k+1})/\mu. \ea\right.\ee

Before we prove the global convergence result, we need to prove the following lemma.

\begin{lemma}\label{lem:Fejer-monotone}
Assume that $(R^*,W^*)$ is an optimal solution of \eqref{prob:generic-min-sum-2} and $\Lambda^*$ is the corresponding optimal dual variable associated with the equality constraint $R+AW=0$. Assume the step size $\tau$ of the proximal gradient step satisfies $0<\tau<1/2$. Then there exists $\eta>0$ such that the sequence $(R^k,W^k,\Lambda^k)$ produced by \eqref{alg:ADM-generic} satisfies
\be \label{lem:conclusion-eq} \|U^{k}-U^*\|_H^2 - \|U^{k+1}-U^*\|_H^2 \geq \eta\|U^k - U^{k+1}\|_H^2, \ee where $U^* = \begin{pmatrix} W^* \\ \Lambda^* \end{pmatrix}$, $U^k = \begin{pmatrix} W^k \\ \Lambda^k \end{pmatrix}$ and $H = \begin{pmatrix} \frac{1}{\mu\tau} I_{p\times p} & 0 \\ 0 & \mu I_{p\times p} \end{pmatrix}$, and the norm $\|\cdot\|_H^2$ is defined as $\|U\|_H^2 = \langle U, HU\rangle$ and the corresponding inner product $\langle \cdot,\cdot\rangle_H$ is defined as $\langle U,V\rangle_H = \langle U,HV\rangle$.
\end{lemma}
\begin{proof}
Since $(R^*,W^*,\Lambda^*)$ is optimal to \eqref{prob:generic-min-sum-2}, it follows from the KKT conditions that the followings hold:
\be \label{lem-KKT-x} 0 \in \partial F(R^*) - \Lambda^*, \ee
\be \label{lem-KKT-y} 0 \in \partial G(W^*) - A^\top\Lambda^*, \ee
and
\be \label{lem-KKT-feas} 0 = R^*+AW^*.\ee

Note that the first-order optimality conditions for the first subproblem (i.e., the subproblem with respect to $R$) in \eqref{alg:ADM-generic}
are given by
\be \label{lem-OPTcond-x} 0\in \partial F(R^{k+1})-\Lambda^k+\frac{1}{\mu}(R^{k+1}+AW^k-0). \ee
By using the updating formula for $\Lambda^k$, i.e.,
\be\label{update-lambda} \Lambda^{k+1} = \Lambda^k - (R^{k+1}+AW^{k+1})/\mu, \ee
\eqref{lem-OPTcond-x} can be reduced to
\be \label{lem-OPTcond-x-reduced} 0 \in \partial F(R^{k+1}) -\Lambda^{k+1} - \frac{1}{\mu}(AW^{k+1}-AW^k).\ee
Combining \eqref{lem-KKT-x} and \eqref{lem-OPTcond-x-reduced} and using the fact that $\partial F(\cdot)$ is a monotone operator, we get
\be\label{lem-OPTcond-x-final} \langle R^{k+1}-R^*, \Lambda^{k+1}-\Lambda^*+\frac{1}{\mu}(AW^{k+1}-AW^k) \rangle \geq 0. \ee

The first-order optimality conditions for the second subproblem (i.e., the subproblem with respect to $W$) in \eqref{alg:ADM-generic} are given by
\be \label{lem-OPTcond-y} 0\in\partial G(W^{k+1}) + \frac{1}{\mu\tau}(W^{k+1}-(W^k-\tau A^\top(AW^k+R^{k+1}-\mu\Lambda^k))). \ee
Using \eqref{update-lambda}, \eqref{lem-OPTcond-y} can be reduced to
\be \label{lem-OPTcond-y-reduced} 0\in \partial G(W^{k+1})+\frac{1}{\mu\tau}(W^{k+1}-W^k+\tau A^{\top}(AW^k-AW^{k+1}-\mu\Lambda^{k+1})). \ee
Combining \eqref{lem-KKT-y} and \eqref{lem-OPTcond-y-reduced} and using the fact that $\partial G(\cdot)$ is a monotone operator, we get
\be \label{lem-OPTcond-y-final} \langle W^{k+1}-W^*, \frac{1}{\mu\tau}(W^k-W^{k+1})-\frac{1}{\mu}A^\top(AW^k-AW^{k+1})+A^\top\Lambda^{k+1}-A^\top\Lambda^* \rangle\geq 0.\ee

Summing \eqref{lem-OPTcond-x-final} and \eqref{lem-OPTcond-y-final}, and using $R^*=-AW^*$ and $R^{k+1}=\mu(\Lambda^k-\Lambda^{k+1})-AW^{k+1}$, we obtain,
\be \label{lem-proof-inequa-1}
\frac{1}{\mu\tau}\langle W^{k+1}-W^*,W^k-W^{k+1} \rangle + \mu\langle\Lambda^{k+1}-\Lambda^*,\Lambda^k-\Lambda^{k+1} \rangle \geq\langle \Lambda^k-\Lambda^{k+1},AW^k-AW^{k+1}\rangle.
\ee
Using the notation of $U^k$, $U^*$ and $H$, \eqref{lem-proof-inequa-1} can be rewritten as
\be\label{lem-proof-inequa-2}
\langle U^{k+1}-U^*,U^k - U^{k+1} \rangle_H \geq \langle\Lambda^k-\Lambda^{k+1},AW^k-AW^{k+1}\rangle,\ee which can be further written as
\be\label{lem-proof-inequa-3} \langle U^k-U^*,U^k - U^{k+1} \rangle_H \geq \|U^k-U^{k+1}\|_H+\langle\Lambda^k-\Lambda^{k+1},AW^k-AW^{k+1}\rangle.\ee
Combining \eqref{lem-proof-inequa-3}
with the identity \[\|U^{k+1}-U^*\|_H^2=\|U^{k+1}-U^k\|_H^2-2\langle U^{k}-U^{k+1}, U^k-U^*\rangle_H+\|U^k-U^*\|_H^2,\]
we get \be \label{lem-proof-inequa-4}
\begin{array}{ll} & \|U^k-U^*\|_H^2 - \|U^{k+1}-U^*\|_H^2 \\ = & 2\langle U^k-U^{k+1},U^k-U^*\rangle_H-\|U^{k+1}-U^k\|_H^2 \\
\geq & \|U^{k+1}-U^k\|_H^2 + 2\langle\Lambda^k-\Lambda^{k+1},AW^k-AW^{k+1}\rangle.\end{array} \ee

Let $\xi:=\tau+1/2$, then we know that $2\tau < \xi<1$ since $0<\tau<1/2$. Let $\rho:=\mu\xi$. Then from Cauchy-Schwartz inequality we have
\be\label{lem-proof-inequa-5} \ba{lll} 2\langle\Lambda^k-\Lambda^{k+1},AW^k-AW^{k+1}\rangle & \geq &  -\rho\|\Lambda^k-\Lambda^{k+1}\|^2-\frac{1}{\rho}\|AW^k-AW^{k+1}\|^2 \\ & \geq & -\rho\|\Lambda^k-\Lambda^{k+1}\|^2-\frac{1}{\rho}\lambda_{\max}(A^\top A)\|W^k-W^{k+1}\|^2 \\ & = & -\rho\|\Lambda^k-\Lambda^{k+1}\|^2-\frac{2}{\rho}\|W^k-W^{k+1}\|^2,\ea\ee where the $\lambda_{\max}(A^\top A)$ denotes the largest eigenvalue of matrix $A^\top A$ and the equality is due to the fact that $\lambda_{\max}(A^\top A)=2$.
Combining \eqref{lem-proof-inequa-4} and \eqref{lem-proof-inequa-5} we get
\be\label{lem-proof-inequa-6} \ba{lll}\|U^k-U^*\|_H^2 - \|U^{k+1}-U^*\|_H^2 & \geq & (\frac{1}{\mu\tau}-\frac{2}{\rho})\|W^k-W^{k+1}\|^2+(\mu-\rho)\|\Lambda^k-\Lambda^{k+1}\|^2 \\ & \geq & \eta\|U^k-U^{k+1}\|_H^2,\ea\ee where $\eta:=\min\{1-\frac{2\mu\tau}{\rho},1-\frac{\rho}{\mu}\}=\min\{1-\frac{2\tau}{\xi},1-\xi\}>0$. This completes the proof.
\end{proof}

We are now ready to give the main convergence result of Algorithm \eqref{alg:ADM-generic}.
\begin{theorem}\label{the:main-convergence}
The sequence $\{(R^k,W^k,\Lambda^k)\}$ produced by Algorithm \eqref{alg:ADM-generic}
from any starting point converges to an optimal solution to Problem \eqref{prob:generic-min-sum-2}.
\end{theorem}
\begin{proof}
From Lemma \ref{lem:Fejer-monotone} we can easily get that
\begin{itemize}
\item (i) $\|U^k-U^{k+1}\|_H \rightarrow 0$;
\item (ii) $\{U^k\}$ lies in a compact region;
\item (iii) $\|U^k-U^*\|_H^2$ is monotonically non-increasing and thus converges.
\end{itemize}
It follows from (i) that $\Lambda^k-\Lambda^{k+1}\rightarrow 0$ and $W^k-W^{k+1}\rightarrow 0$.
Then \eqref{update-lambda} implies that $R^k-R^{k+1}\rightarrow 0$ and $R^k+AW^k\rightarrow 0$.
From (ii) we obtain that, $U^k$ has a subsequence $\{U^{k_j}\}$ that converges to $\hat{U}=(\hat{W},\hat\Lambda)$,
i.e., $\Lambda^{k_j}\rightarrow\hat\Lambda$ and $W^{k_j}\rightarrow\hat{W}$. From $R^k+AW^k\rightarrow 0$ we also get
that $R^{k_j}\rightarrow\hat{R}:=-A\hat{W}$. Therefore, $(\hat{R},\hat{W},\hat{\Lambda})$ is a limit point of $\{(R^k,W^k,\Lambda^k)\}$.

Note that \eqref{lem-OPTcond-x-reduced} implies that
\be \label{the-convergence-inequa-2} 0 \in \partial F(\hat{R})-\hat{\Lambda}.\ee
Note also that \eqref{lem-OPTcond-y-reduced} implies that
\be \label{the-convergence-inequa-3} 0 \in \partial G(\hat{W}) - A^\top\hat\Lambda.\ee
\eqref{the-convergence-inequa-2}, \eqref{the-convergence-inequa-3} and $\hat{R}+A\hat{W}=0$
imply that $(\hat{R},\hat{W},\hat{\Lambda})$ satisfies the KKT conditions for \eqref{prob:generic-min-sum-2} and thus is an optimal
solution to \eqref{prob:generic-min-sum-2}.
Therefore, we showed that any limit point of $\{(R^k,W^k,\Lambda^k)\}$ is an optimal solution to \eqref{prob:generic-min-sum-2}.

To complete the proof, we need to show that the limit point is unique. Let $\{(\hat{R}_1,\hat{W}_1,\hat{\Lambda}_1)\}$ and $\{(\hat{R}_2,\hat{W}_2,\hat{\Lambda}_2)\}$ be any two limit points of $\{(R^k,W^k,\Lambda^k)\}$. As we have shown, both $\{(\hat{R}_1,\hat{W}_1,\hat{\Lambda}_1)\}$ and $\{(\hat{R}_2,\hat{W}_2,\hat{\Lambda}_2)\}$ are optimal solutions to \eqref{prob:generic-min-sum-2}. Thus, $U^*$ in \eqref{lem-proof-inequa-6} can be replaced by $\hat{U}_1:=(\hat{R}_1,\hat{W}_1,\hat{\Lambda}_1)$ and $\hat{U}_2:=(\hat{R}_2,\hat{W}_2,\hat{\Lambda}_2)$. This results in
\[\|U^{k+1}-\hat{U}_i\|_H^2\leq \|U^k-\hat{U}_i\|_H^2, \quad i = 1,2,\]
and we thus get the existence of the limits
\[\lim_{k\rightarrow\infty}\|U^k-\hat{U}_i\|_H = \eta_i < +\infty, \quad i = 1,2.\]
Now using the identity
\[\|U^k-\hat{U}_1\|_H^2 - \|U^k-\hat{U}_2\|_H^2 = -2\langle U^k, \hat{U}_1-\hat{U}_2\rangle_H + \|\hat{U}_1\|_H^2 - \|\hat{U}_2\|_H^2\]
and passing the limit we get
\[\eta_1^2 - \eta_2^2 = -2\langle \hat{U}_1, \hat{U}_1-\hat{U}_2\rangle_H + \|\hat{U}_1\|_H^2 - \|\hat{U}_2\|_H^2 = -\|\hat{U}_1-\hat{U}_2\|_H^2\]
and
\[\eta_1^2 - \eta_2^2 = -2\langle \hat{U}_2, \hat{U}_1-\hat{U}_2\rangle_H + \|\hat{U}_1\|_H^2 - \|\hat{U}_2\|_H^2 = \|\hat{U}_1-\hat{U}_2\|_H^2.\]
Thus we must have $\|\hat{U}_1-\hat{U}_2\|_H^2=0$ and hence the limit point of $\{(R^k,W^k,\Lambda^k)\}$ is unique.
\end{proof}

We now immediately have the global convergence result for Algorithm \ref{alg:pgadm} for solving Problem \eqref{prob:latent-variable}.
\begin{corollary}
The sequence $\{(R^k,S^k,L^k,\Lambda^k)\}$ produced by Algorithm \ref{alg:pgadm}
from any starting point converges to an optimal solution to Problem \eqref{prob:latent-variable}.
\end{corollary}

\end{document}